\documentclass[a4paper,twoside,reqno]{amsart}
\usepackage{amsmath}
\usepackage{amsfonts}
\usepackage{amsthm}
\usepackage{amssymb}
\usepackage{verbatim}
\usepackage{lingmacros}
\usepackage{hyperref}
\usepackage{comment}
\usepackage{nth}
\usepackage[english]{babel}
\usepackage{graphicx}
\usepackage{wrapfig}
\usepackage{footmisc}
\usepackage{ifpdf}
\ifpdf
\fi
\usepackage{tikz}
\usepackage{tikz-cd} 
\newcommand{\N}{\mathbb N}
\newcommand{\C}{\mathbb C}
\newcommand{\Z}{\mathbb Z}
\newcommand{\R}{\mathbb R}
\newcommand{\PP}{\mathbb P}
\newcommand{\A}{\mathbb A}
\newcommand{\Q}{\mathbb Q}
\newcommand{\F}{\mathcal{F}}
\renewcommand{\O}{\mathcal{O}}
\newcommand{\chapter}{\section}

\newcommand{\Int}{\operatorname{Int}}
\newcommand{\is}{\operatorname{IS}}
\newcommand{\Hom}{\operatorname{Hom}}
\newcommand{\Spec}{\operatorname{Spec}}
\newcommand{\codim}{\operatorname{codim}}
\newcommand{\pt}{\operatorname{pt}}
\newcommand{\ch}{\operatorname{ch}}
\newcommand{\td}{\operatorname{td}}
\newcommand{\Td}{\operatorname{Td}}
\newcommand{\var}{\operatorname{var}}
\newcommand{\mult}{\operatorname{mult}}
\newcommand{\id}{\operatorname{id}}

\newtheorem{theorem}{Theorem}[section]

\newtheorem{lemma}[theorem]{Lemma}
\theoremstyle{remark}
\newtheorem{remark}[theorem]{Remark}

\newtheorem{example}[theorem]{Example}
\theoremstyle{definition}

\title[Positivity of local equivariant Hirzebruch class for toric varieties]{The positivity of local equivariant Hirzebruch class for toric varieties}
\author{Kamil Rychlewicz}
\address{Faculty of Mathematics, Informatics and Mechanics, University of Warsaw; 
\protect\linebreak
Institute of Science and Technology Austria, am Campus 1, 3400 Klosterneuburg, \nobreak{Austria},
{\tt kamil.rychlewicz@ist.ac.at}
}
\begin{document}
\begin{abstract}
The central object of investigation of this paper is the Hirzebruch class, a deformation of the Todd class, given by Hirzebruch (for smooth varieties) in his celebrated book \cite{Hirzebruch}. The generalization for singular varieties is due to Brasselet-Sch\"{u}rmann-Yokura. Following the work of Weber, we investigate its equivariant version for (possibly singular) toric varieties. The local decomposition of the Hirzebruch class to the fixed points of the torus action and a formula for the local class in terms of the defining fan are mentioned. After this review part (sections \ref{firstsec}--\ref{lastsec}), we prove the positivity of local Hirzebruch classes for all toric varieties (Theorem \ref{final}), thus proving false the alleged counterexample given by Weber.
\end{abstract}

\maketitle 

\section{Acknowledgments}

This paper is an abbreviation of the author's Master's Thesis written at Faculty of Mathematics, Informatics and Mechanics of University of Warsaw under the supervision of Andrzej Weber. The author owes many thanks to the supervisor, for introducing him to the topic and stating the main problem, for pointing the important parts of the general theory and for the priceless support regarding editorial matters.

\chapter{Introduction}

The Todd class, which appears in the formulations of the celebrated Hirzebruch-Riemann-Roch theorem, and more general Grothendieck-Riemann-Roch theorem, is originally defined for smooth complete varieties (or schemes) only. In \cite{BFM} Baum, Fulton and MacPherson gave a generalization for singular varieties, which in general is forced to lie in homology instead of cohomology. For simplicial toric varieties, an equivariant version of Hirzebruch-Riemann-Roch theorem was given by Brion and Vergne in \cite{BV} with a purely combinatorial approach, without employing the general algebro-geometric theory. In this case, the equivariant Todd class may actually be defined in cohomology. However, it has to be defined in homology for general toric varieties -- which is done by Brylinski and Zhang in \cite{BZ}.

An extension of the Todd genus, the $\chi_y$-genus of a smooth variety $X$ (and, in general, of any locally free sheaf) was defined by Hirzebruch in \cite{Hirzebruch}, as a polynomial of one variable $y$ that specializes to Todd genus under substitution $y=0$. Analogously, the full Todd class extends to the Hirzebruch class, defined in $H^*(X,\Q)[y]$. It specializes to the Todd class for $y=0$ and its highest component integrates to the $\chi_y$-genus.

In \cite{BSY} Brasselet, Sch\"{u}rmann and Yokura gave a strong generalization of the Hirzebruch class not only for singular varieties, but for morphisms of them: $Td_y(f:X\to M)\in H_*^{BM}(M,\Q)[y]$. Again, it is forced to lie in homology, and for $f=\id_X$ with $X$ smooth it reduces to the Hirzebruch class (after applying Poincar\'{e} duality) and for $M=\pt$ it reduces to $\chi_y(X)$.

This construction was investigated by Maxim and Sch\"{u}rmann in \cite{MaSc} with relation to toric varieties and later by Weber in \cite{Weber2016} in equivariant setting. Like for the Todd class in \cite{BV}, the formulas in terms of defining fans and classes of orbits were derived. By the localization theorem of Atiyah-Bott or Berline-Vergne, the computation is split to computations of local Hirzebruch classes for fixed points.

In this paper we give a positive answer to the question posed by Weber (\cite{Weber2016}) about positivity of local Hirzebruch classes for toric varieties. To this end, we employ the Weber's result for simplicial cones and combine it with inclusion-exclusion-type equality for triangulations of arbitrary polyhedral cones.

The positivity problem is parallel to analogous results for Thom polynomials and Chern classes, as well as the Huh's result of positivity of Chern-Schwartz-MacPherson class of Grassmannians.
There is still some space for further investigation, e.g. whether the local positivity holds in general for rational singularities.

\section{Notation}\label{firstsec}

The set of natural numbers $\N$ contains $0$. All the algebraic varieties we consider are defined over $\C$. All toric varieties are normal by definition. For any integral domain $R$, by $F(R)$ we denote its fraction field.

For any lattice $L$, i.e. a group isomorphic to $\Z^n$ for some integer $n$, by $L_\Q$ we denote $L\otimes \Q$ (and similarly for $\R$ instead of $\Q$). For a vector space $V$ and its subset $A$, by $A^\perp$ we denote the annihilator of $A$, i.e. $\{\phi \in V^*: \forall_{a\in A} \phi(a) = 0\}$. For a vector space $V$, $S^*(V)$ is the symmetric algebra of $V$, which is canonically isomorphic to the (graded) ring of polynomial functions on $V^*$.

For a torus $T$ we shall use the notation $M = \Hom(T,\C^*)$ for its character group and $N = \Hom(\C^*,T)$ for the group of its one-parameter subgroups. 

For a rational polyhedral cone $\sigma\subset N_{\R}$ by $\sigma^\vee$ we denote the dual cone $\sigma^\vee = \{ v \in M_\R : \langle v,w \rangle \ge 0 \text{ for all }w\in \sigma\}$. Then for a strongly convex rational polyhedral cone $\sigma\subset N_{\R}$ by $X_\sigma$ we denote the toric variety defined by it, i.e $X_\sigma = \Spec \C[ \sigma^\vee \cap M]$. Any set of generators $w_1$, $w_2$, \dots, $w_K$ of $\sigma^\vee\cap M$ defines an embedding $X_\sigma \to \C^K$.

If we now consider a fan $\Sigma$ over $N$, every cone $\sigma\in\Sigma$ gives rise to an inclusion $X_\sigma\hookrightarrow X_\Sigma$ as an open subset $U_\sigma$. Each of $U_\sigma$ contains a unique closed orbit of $T$ which (following notation from \cite{BV}) we denote by $\Omega_\sigma$.
It satisfies the equality $\dim \Omega_\sigma = \codim_{N_\R} \sigma$. For a fan $\Sigma$ by $\Sigma'$ we denote the set of the maximal cones.

Now suppose that we are given a smooth complete variety $X$ with an action of a torus $T$. The fibration $ET\times^T X \to BT$ (with fiber X) gives rise to the push-forward map (\cite{AB})
\begin{equation}\label{push}
\int_X: H_T^*(X) \to H_T^*(\pt),
\end{equation}
which is a homomorphism of $H_T^*(\pt)$-modules, decreasing grading by $2\dim_{\C}(X)$. This is the equivariant version of the classical integration of differential forms on an oriented (real) manifold $X$, which maps $H^*(X) \to H^*(\pt)$. The latter is nontrivial on $H^{\dim_\R(X)}(X)$ only, but the equivariant version may on contrary be nontrivial in higher gradations.

\section{Equivariant homology}

We shall make use of the equivariant cohomology throughout this paper, but also of the equivariant \emph{homology} as well. As the construction is less standard, we sketch it here. For details, we refer the reader to \cite{Ohmoto}, \cite{BZ} or to \cite{Edidin} for the original construction of equivariant Chow groups on which the equivariant homology is based. Suppose we are given a complex reductive group $G$. Then we may consider the Totaro's approximation (\cite{Edidin2}, \cite{Totaro}) $E_n \to B_n$ of the universal bundle $EG\to BG$. Using the approximation, we define the equivariant homology via Borel-Moore homology:

$$H_i^G(X) = H_{i+2d_n-2g}^{BM} (E_n \times^G X,\Q),$$
where $g = \dim_\C (G)$, $d_n = \dim_\C (E_n)$. It turns out (\cite{Edidin}) that $H_i^G(X)$ is independent of the choice of $n$, provided $n$ is large enough. It may happen that $H_i^G(X)$ is nonzero for negative $i$. In fact, $H_*^G(X)$ is endowed with the cap product $\cap: H_G^n(X) \otimes H_i^G(X) \to H_{i-n}^G(X)$, which makes $H_*^G(X)$ into a $H_G^*(X)$-module, but with the grading reversed.

For any complex variety $X$ of complex dimension $n$ the inclusion of its smooth locus $i:X_{smooth} \to X$ induces isomorphism $i^*:H_{2n}^{BM}(X) \to H_{2n}^{BM}(X_{smooth})$ (cf. \cite[Lemma 19.1.1]{Fulton2}), which makes it possible to define the orientation class $[X] \in H_{2n}^{BM} (X)$ as the preimage of the orientation class $[X_{smooth}] \in H_{2n}^{BM}(X_{smooth})$ (note that as we deal with the Borel-Moore homology, we do not need $X_{smooth}$ to be compact).
Then if $X$ is smooth, the Poincar\'{e} duality (i.e. cap product with the fundamental class $[X]\in H_{2d}^G(X)$) provides an isomorphism $H_i^G(X) \simeq H_G^{2d-i}(X)$ for $d=\dim_\C (X)$, thus showing that $H_*^G(X) \simeq H_G^*(X)$ as graded $H_G^*(X)$-modules, up to the grading reversion. For example, for $G = \left(\C^*\right)^d$ we have $H^G_*(\pt)\simeq \Q[t_1,\dots,t_n]$ and $t_1$, $t_2$, \dots, $t_n$ all belong to $H^G_{-2}(\pt)$, thus the equivariant homologies of the point exist in nonpositive dimensions only.

\section{Characteristic classes of singular varieties}

\subsection{Todd class}
The Todd class (see \cite{Hirzebruch}, \cite{Milnor}) $\td(X)$, which plays substantial role in the celebrated Hirzebruch-Riemann-Roch theorem and more general Grothendieck-Riemann-Roch theorem is originally defined for smooth varieties using the tangent bundle. It was generalized to singular varieties by Baum-Fulton-MacPherson in \cite{BFM} and satisfies a generalized Grothendieck-Riemann-Roch theorem (proved therein). However, the Baum-Fulton-MacPherson Todd class lies in (Borel-Moore) homology or in Chow groups of a variety $X$, unlike the standard Todd class, which is cohomological. For smooth varieties, the Poincar\'{e} duality enables one to identify them, but we shall make a distinction and denote the cohomological class by $\td(X)$ and the homological one by $\Td(X)$. For details of the construction, we refer the reader to \cite{BFM}.

\subsection{Hirzebruch class}\label{hirzsection}

The Hirzebruch's $\chi_y$-genus is defined in \cite[\S 15.5]{Hirzebruch} for a compact complex manifold $X$. It is equal to 
\begin{equation}
\label{hirz}
\int_X \td(X) \ch(\Omega_X^y)
\end{equation}
where $\Omega_X^y$ is defined as the formal sum  $\bigoplus\limits_{p=0}^{\dim_\C(X)}\Omega_X^p y^p$, with $\Omega_X^p$ being the sheaf of holomorphic $p$-forms on $X$. Now the integrand in \eqref{hirz} is called the \emph{Hirzebruch class} and denoted by $\td_Y(X) \in H^*(X)[y]$. The Hirzebruch class reduces to the Todd class under the substitution $y=0$. In \cite{Hirzebruch} Hirzebruch gave a generalization of Hirzebruch-Riemann-Roch theorem, which uses the Hirzebruch class instead of Todd class and generalized Chern character $\ch_y$ (\S 12.2 ibid.) to compute the $\chi_y$-characteristic of a bundle (\S 15.5 ibid.).

A generalization of the Hirzebruch class to all, possibly singular varieties was given by {\nobreak Brasselet}, Sch{\"u}rmann and Yokura in \cite{BSY}. Just like the Baum-Fulton-MacPherson Todd class mentioned above, the generalized Hirzebruch class belongs to Borel-Moore homology instead of cohomology. Precisely, the \emph{motivic} Hirzebruch class is defined for varieties over any given variety $X$, yielding a map $\Td_y:K_0(\var/X) \to H_*^{BM}(X)\otimes \Q[y]$. Here $K_0(\var/X)$ is the \emph{Grothendieck group of varieties over $X$} (cf. \cite{Lo},\cite{Bi}).

The motivic Hirzebruch class generalizes the (homological) Hirzebruch class of smooth varieties, i.e. if $X$ is smooth and complete and we consider $\id_X:X\to X$, then $\Td_y(\id_X) = \td_y(X)\cap [X] \in H_*(X)\otimes \Q[y]$. The motivic class also satisfies the naturality condition, i.e. for $f:X\to Y$ and a proper morphism $g:Y\to Z$
$$\Td_y(g\circ f: X\to Y\to Z) = g_* \Td_y(f:X\to Y). $$
In particular, if we consider $f = \id_X$ and $g:X\to\pt$ for smooth and complete variety $X$, then we get $\chi_y(X) = \Td_y(X\to\pt)$ under identification $H_0(X)\simeq H_0(\pt)$. Thus we can compute the $\chi_y$-genus using the map $X\to\pt$. In particular, it means that if $X$ can be split into a finite (set-theoretic) disjoint sum of constructible sets
$$X = \bigsqcup_{i=1}^n X_i,$$
then we get
\begin{equation}
\label{splitchi}
    \chi_y(X) = \sum_{i=1}^n \chi_y(X_i).
\end{equation}

As a special case of the motivic Hirzebruch class, for any (possibly singular) variety $X$ and identity $\id_X$ we get the (absolute) Hirzebruch class $\Td_y(id_X)\in H_*(X)$ of a variety $X$, which generalizes the homological Hirzebruch class $\td_y(X)\cap [X]$ for smooth complete varieties. In the same way, we can define the $\chi_y$-genus of any variety $X$ as $\Td_y(X\to\pt)\in H_0(\pt)\otimes \Q[y]\simeq \Q[y]$. So defined Hirzebruch class $\Td_y(\id_X)$ may not, in general, reduce to the Baum-Fulton-MacPherson's Todd class. It does, however, for a vast class of varieties (at least for varieties with du Bois singularities, cf. \cite[Example 3.2]{BSY}), which includes toric varieties.

\section{Todd class for toric varieties}

In \cite{BV} Brion and Vergne proved an equivariant Riemann-Roch theorem for complete simplicial toric varieties without appealing to general theory. In doing this, they introduced the equivariant Todd class for complete simplicial toric varieties, providing it with an explicit formula. The results were generalized in \cite{BZ} for any complete toric variety. We include here a brief summary of the results.

Suppose we are given a $T$-linearized coherent sheaf $\F$ on a simplicial toric variety $X$ with an action of a $d$-dimensional torus. \emph{$T$-linearized coherent sheaf} means a coherent sheaf endowed with an action of $T$, which is compatible with $\O_X$-module structure on $\F$. Then in \cite{BV} the \emph{Chern character} $\ch^T(\F)$ of $\F$ is defined in the completion $\hat{H}^*_T(X)$ of $H^*_T(X)$.

If $\F$ is locally free, it defines a $T$-equivariant vector bundle on $X$, which can be pulled back to $ET \times X$ and then divided by the diagonal action of $T$, yielding a bundle over $ET\times^T X$. The usual Chern character of this bundle then coincides with the one mentioned above. The \emph{equivariant Todd class} $\td^T(X) \in \hat{H}_T(X)$ is defined for any complete simplicial toric variety $X$ such that the following equivariant Riemann-Roch theorem (\cite[Theorem 4.1]{BV}) is satisfied:

\begin{theorem}
 For any $T$-linearized coherent sheaf $\F$ on a complete simplicial toric variety $X$ the following holds
 $$\chi^T(X,\F) = \int_X \ch^T(\F) \td^T(X),$$
where $\chi^T(\F)$ is the \emph{equivariant Euler characteristic} of $\F$ as defined in \cite[1.3]{BV}.
\end{theorem}

The formulas for $\td^T(X)$ are given in terms of the fan $\Sigma$ defining $X$. Before stating them, we first outline the rules of summing infinite series as provided in \cite[1.3]{BV}. Let $\Z[M]$ be the group ring of $M$ over $\Z$ and let $\Z[[M]]$ be the $\Z[M]$-module of all formal power series
$$f(e) = \sum_{m\in M} a_m e^m,$$
possibly infinite in every direction. We call such $f$ \emph{summable} if there exists $P(e)\in \Z[M]$ and a finite sequence $m_i\in M\setminus\{0\}$ such that
$$f(e)\prod (1-e^{m_i}) = P(e).$$
Then we call the element $\frac{P}{\prod (1-e^{m_i})}\in F(\Z[M])$ the sum of $f(e)$. It is easy to see that this element does not depend on the choice of $P$ and $m_i$'s. As an example, we have
$$\sum_{k=0}^\infty e^{km} = \frac{1}{1-e^m}$$
and, in general, if $m_1, m_2, \dots, m_n\in M$ are linearly independent, then
\begin{equation}
\label{ogolnasuma}
    \sum_{k_1,k_2,\dots,k_n\in \N} e^{\sum_{i=1}^n k_i m_i} = \prod_{i=1}^n \frac{1}{1-e^{m_i}}.
\end{equation}

Now by a result from \cite{BV}, for any $d$-dimensional cone $\sigma\in\Sigma$ the restriction of the Todd class to $X_\sigma$, which is a series in $\hat{S}(M_\Q)$ (the completion of $S^*(M_\Q)$), is given by

$$\td_\sigma^T(X) = \mult\left(\sigma^\vee\right)^{-1} \prod_{i=1}^d (-m_i) \sum_{m\in\sigma^\vee} e^m,$$
with $m_i$ being the generators of rays of $\sigma^\vee$ and $\mult(\sigma^\vee)$ being the index of lattice $\bigoplus\limits_{i=1}^d \Z m_i$ in $\Z\sigma^\vee$. By definition $e^m = \sum\limits_{i=0}^\infty \frac{m^i}{i!}$ and hence the series $\sum\limits_{m\in\sigma^\vee} e^m$ is not summable in $\hat{S}(M_\Q)$ as a series in $m$. It has to be considered as an element of  $\Z[[M]]$ and summed according to the rules given above:

$$\sum_{m\in\sigma^\vee} e^m = \left(\sum_{m\in C\cap M} e^m\right) \prod_{i=1}^d \frac{1}{1-e^{m_i}},$$
where $C$ is the cube $\left\{\sum\limits_{i=1}^d \theta_i m_i : \theta_i\in [0,1)\right\}$.

The sum can then be mapped to a Laurent series in $F(\hat{S}(M_\Q))$ by mapping $e^m$ to $\sum\limits_{i=0}^\infty \frac{m^i}{i!}$. It will contain factors of the form 

$$\frac{1}{1-e^{m_i}} = \frac{1}{m_i} \cdot \frac{1}{-1-\frac{m_i}{2}-\frac{m_i^2}{6}-\dots}.$$
After multiplying it by $\prod_{i=1}^d (-m_i)$ we get an element of $\hat{S}(M_\Q)$.

In \cite{BZ} and \cite{Edidin} the homological equivariant Todd class $\Td^T(X)$, which is equal to $\td^T(X) \cap [X]$ for $[X]$ being the orientation class of $X$ (under the assumption of smoothness of $X$), is defined for any $T$-variety via Baum-Fulton-MacPherson's Todd class.

The formula for so defined equivariant homological Todd class of any complete toric variety is given in \cite[Theorem 9.4]{BZ}. By $L$ we denote the multiplicative system in $H^*_T(\pt)$ generated by $m\in M\setminus\{0\} \subset S^*(M) = H^*_T(\pt)$. Then we have:

\begin{theorem}\label{Brylinski}
 For a complete toric variety $X=X_\Sigma$ we have
 $$L^{-1}\Td^T(X) = \sum_{\sigma\in\Sigma'}\sum_{m\in\sigma^\vee} e^m [\Omega_\sigma].$$
\end{theorem}

The left-hand side is the image of $\Td^T(X)$ in $L^{-1}\hat{H}_*^T(X)$. For the same reasons that were described above, the localization is needed to make sure the right-hand side summands make sense, as the sums of type $\sum_{m\in\sigma^\vee} e^m$ contain factors of the form 
$$\frac{1}{1-e^m} = \frac{1}{m} \cdot \frac{1}{-1-\frac{m}{2}-\frac{m^2}{6}-\dots}$$
for $m\in M$. Thus the right-hand side only makes sense when we localize to $L^{-1}\hat{H}_*^T(X)$. As a part of the theorem, however, the sum actually lies in $\hat{H}_*^T(X)$.

For a complete simplicial toric variety there exists the Poincar\'{e} duality (see \cite[Lemma 8.10]{BZ}) -- an isomorphism of graded $H^*_G(\pt)$-modules $PD : R_\Sigma \to H_*(X)$ ($PD:(R_\Sigma)_k \simeq H_{2d-2k}^T(X)$), which generalizes the usual Poincar\'{e} duality $\omega \mapsto \omega\cap [X]$ for smooth varieties. For any $\sigma\in\Sigma'$ with $m_1$, $m_2$, \dots, $m_d$ being the generators of rays of $\sigma^\vee$ we may consider the polynomial
$$\phi_{\sigma} = (-1)^d \mult(\sigma^\vee)^{-1} \prod_{i=1}^d m_i$$
defined on $\sigma$ and extend it as zero to other cones, getting $\phi_\sigma \in R_\Sigma$. Then $PD(\phi_\sigma) = [\Omega_\sigma]$, thus Theorem \ref{Brylinski} agrees with the formula for $\Td_\sigma^T(X)$ given above.

\section{Hirzebruch class for toric varieties}\label{lastsec}

The Hirzebruch class was defined in motivic version for singular varieties in \cite{BSY}, as described in \S \ref{hirzsection}. We are now concerned with the computations for toric varieties. For the nonequivariant version, it is easy to observe that for a torus $(\C^*)^d$ its $\chi_y$ genus is equal to $(-1-y)^d$ and then, using the additivity in \eqref{splitchi} 
we can easily express the $\chi_y$-genus of a toric variety $X_\Sigma$ as
$$\sum_{\sigma\in\Sigma} (-1-y)^{\codim \sigma}.$$

The full Hirzebruch class $\Td_y(\id_X)$ is more challenging. In \cite{MaSc} formulas for the Hirzebruch class was developed in the case of toric varieties. We are interested in the equivariant version, which was introduced in \cite{Weber2016}. Just like the Todd class defined in \cite{BZ} for general toric setting, the Hirzebruch class is now an element of (equivariant) homology instead of cohomology.

Now we include a short summary of the results from \cite{Weber2016}, regarding the equivariant Hirzebruch class for toric varieties. Suppose we are given an equivariant map $X\to M$ of (possibly singular) varieties with torus action. Then, analogously to the equivariant Todd class, we can consider the equivariant Hirzebruch class
$$\Td^T_y(X\to M) \in \hat{H}^T_*(M)[y].$$
The completion of $H^T_*(M)$ has to be understood as taking the product of all gradations. 

So defined equivariant $\chi_y$-genus (for $M=\pt$) does not, in fact, yield any new information about $X$ (compared to the non-equivariant version). By rigidity theorem (\cite{Musin}, \cite{Weber2016}) $\chi^T_y(X) = \Td^T_y(X\to\pt)$ is nontrivial in the zeroth gradation ($H^T_0(\pt)[y] = \Q[y]$) only, being equal to the nonequivariant genus.

For the full Hirzebruch class, the following result (\cite[Theorem 11.3]{Weber2016}) is proven. Using the Atiyah-Bott (\cite{AB}) or Berline-Vergne (\cite{Berline}) localization formula, the computation of the equivariant Hirzebruch class is reduced to the computation of local classes for the fixed points of torus action. Note that to be able to use the localization theorem and to restrict to fix points, we have to work with cohomology instead of homology. For this sake, we embed the given variety in a smooth one and compute the class of the embedding instead of the identity.

\begin{theorem}
\label{formula}
Let $X_\Sigma$ be a complete toric variety and $X_\Sigma\to M$ an embedding in a smooth variety. For a maximal cone $\sigma\in \Sigma$ the restriction of the Hirzebruch class to the fixed point corresponding to $\sigma$ is equal to
$$\Td^T_y(X_\Sigma\to M)|_{p_\sigma} = \sum_{\tau \prec \sigma} (1+y)^{\codim \tau} \sum_{m\in \Int(\sigma^\vee \cap \tau^\perp)} e^{m} [\Omega_\sigma]. $$
\end{theorem}

The summation of infinite series is again done with compliance to rules outlined in \cite{BV}. For any cone $\sigma$ we use the notation $\Int(\sigma)$ to denote its \emph{relative interior}, i.e. (the set of lattice points in) the topological interior of $\sigma$ in the vector space $\R\sigma$. It amounts to subtracting all the proper faces from a cone. Note that if we substitute $y=0$ to the above theorem, we get exactly the formula from Theorem \ref{Brylinski}.

\begin{remark}
We performed the above computations for an embedding $X\to M$ of $X$ into a smooth variety $M$. There are two caveats we should mention. First, we should know that there exists such $M$, with isolated fixed points of $T$ action. For the computation of the local classes for cones it is however enough to embed only the suitable affine parts of $X$ in smooth varieties. This is always possible, as the affine parts embed, by definition, in the affine spaces.

Second, although the right-hand side in Theorem \ref{formula} is a priori defined as a homology (or cohomology) class of $M$, it is the push-forward of analogous class in homology of $X$. It does not follow from the definitions that it does not depend on the choice of $M$ and embedding $X\to M$, but it is known that is does not. The formula from the theorem demonstrates it for toric varieties.
\end{remark}

\section{Positivity of local Hirzebruch class}

\subsection{Motivation}

As mentioned in \cite{Weber2016},
there is a positivity condition that holds for Thom polynomials of invariant varieties in toric representations. Similarly, it is proven in \cite{Weber2016} that the Hirzebruch class can be represented as a polynomial with nonnegative coefficients for 
simplicial toric varieties (\cite[Theorem 13.1]{Weber2016}):

\begin{theorem}
Let $X$ be a toric variety defined by a $d$-dimensional simplicial cone $\sigma\in N$. Let $w_1$, $w_2$, \dots, $w_K$ be the generators of the dual cone $\sigma^\vee\in M$, with $w_1$, $w_2$, \dots, $w_k$ being the primitive vectors on the rays of $\sigma^\vee$. Then the Hirzebruch class of the inclusion of the open orbit $td_y^T(\Omega_{\{0\}}\to\C^K)$, determined by choice of generators $w_1$, $w_2$, \dots, $w_K$ is of the form
\begin{equation} 
\delta^d \prod_{i=1}^k \frac{1}{S_{w_i}} \cdot P(\{S_{w_i}\}_{i=1,2,\dots,K}) [\Omega_\sigma],
\end{equation}
 where $\delta = -y-1$, $P$ is a polynomial with nonnegative coefficients and $S_{w_i} = e^{w_i} - 1$.
\end{theorem}

A similar theorem is also stated for all three-dimensional cones and a four-dimensional counterexample is given. Unfortunately, the counterexample turns out to be mistaken and the theorem holds for all toric varieties. We prove it in the following section.

\subsection{The general proof}

In the view of Theorem \ref{formula} we need to prove nonnegativeness for sums of lattice points in interiors of arbitrary cones (Theorem \ref{main}). We start by stating the equivalent for \emph{closed} simplicial cones.

Let us remind that a strictly convex rational polyhedral cone over a lattice $L$ is a subset of $L_\R$ of the form

$$\sigma = \{ \alpha_1 v_1 + \alpha_2 v_2 +\dots + \alpha_n v_n\in L_\R : \alpha_1,\alpha_2,\dots,\alpha_n  \in \R_{\ge 0}\}$$
such that $\sigma\cap -\sigma = \{0\}$, i.e. $\sigma$ contains no line through the origin. Whenever we sum over elements of a cone $\sigma$ over a lattice $L$, we mean summation over elements of $\sigma\cap L$.

\begin{lemma}
\label{lemgl}
 Suppose we are given an n-dimensional strictly convex rational simplicial cone $\sigma$ over $\Z^m$. Let $w_1$, $w_2$, \dots, $w_K$ be the generators of $\sigma$ (as a semigroup) with $w_1$, $w_2$, \dots, $w_n$ being the primitive vectors on the rays of $\sigma$. Then
 $$\sum_{t\in \sigma} e^t = (-1)^n \prod_{i=1}^n \frac{1}{S_{w_i}} \cdot P(\{S_{w_i}\}_{i=1,2,\dots,K}),$$
 where $P$ is a polynomial with nonnegative coefficients and $S_{w_i} = e^{w_i} - 1$.
\end{lemma}

\begin{proof}
 As the cone $\sigma$ is simplicial, every point of it can be represented uniquely as a sum
 $\sum\limits_{i=1}^n \alpha_i w_i$
 with nonnegative real numbers $\alpha_i$ and conversely, every lattice point of this form belongs to $\sigma$. We can further write
 $$\sum_{i=1}^n \alpha_i w_i = \sum_{i=1}^n \lfloor \alpha_i \rfloor w_i +\sum_{i=1}^n \{\alpha_i\} w_i,$$
 where $\lfloor x \rfloor$ and $\{x\}$ denote the integer part (the floor) and the fractional part of $x$, respectively. This way we expressed every point of $\sigma$ as a unique sum of a point in the semigroup $\bigoplus\limits_{i=1}^n\N w_i$ and a point in the cube $C = \left\{\sum\limits_{i=1}^n \theta_i w_i : \theta_i\in [0,1)\right\}$. Now it means that (using \eqref{ogolnasuma})
 
 $$ \sum_{t\in \sigma} e^t = \sum_{k_1, k_2, \dots, k_n \in \N}\quad \sum_{u\in C\cap \Z^m} e^{\,\, u + \sum\limits_{i=1}^n k_i w_i}
 = \prod_{i=1}^n \frac{1}{1-e^{w_i}} \sum_{u\in C\cap \Z^m} e^u.$$
 The set $C\cap \Z^m$ is finite and as every $u\in C\cap \Z^m$ is a sum of the generators among $w_1$, $w_2$, \dots, $w_N$, the exponent $e^u$ is a product of factors of the form $e^{w_i} = S_{w_i}+1$. Thus $\sum_{u\in C\cap \Z^m} e^u = P(\{S_{w_i}\}_{i=1,2,\dots,K})$ for some polynomial $P$ with nonnegative coefficients. Summing up, we get
 $$\sum_{t\in \sigma} e^t = \prod_{i=1}^n \left( -\frac{1}{S_{w_i}} \right) \cdot P(\{S_{w_i}\}_{i=1,2,\dots,K}),$$
 as desired.
\end{proof}

Now we are ready to prove the general version.

\begin{theorem}
\label{main}
Suppose we are given any $n$-dimensional strictly convex rational polyhedral cone $\sigma$ over $\Z^n$. Let $w_1$, $w_2$, \dots, $w_K$ be the generators of $\sigma$ with $w_1$, $w_2$, \dots, $w_k$ being the primitive vectors on the rays of $\sigma$. Then
 $$\sum_{t\in \Int \sigma} e^t = (-1)^n \prod_{i=1}^k \frac{1}{S_{w_i}} \cdot P(\{S_{w_i}\}_{i=1,2,\dots,K}),$$
 where $P$ is a polynomial with nonnegative coefficients and $S_{w_i} = e^{w_i} - 1$.
\end{theorem}

\begin{proof}
 The idea is to triangulate the cone and use Lemma \ref{lemgl} for the cones appearing in the triangulation (not only the maximal ones).
 
 Let $\sigma = \bigcup\limits_{i=1}^k \tau_k$ be a triangulation of $\sigma$
 with $\tau_k$ being the simplicial cones of dimension $n$, such that all added rays are rational. 
  Then if $\Sigma$ is the set of all the cones, of all dimensions, appearing in the triangulation (i.e. $\tau_k$'s and their faces, including the trivial cone as the 0-dimensional face of any other), we get
 $$\sigma = \bigsqcup\limits_{\tau\in\Sigma} \Int\tau.$$
 Moreover, $\Int \sigma$ is a subsum of this disjoint sum and also every $\tau\in\Sigma$ is a subsum. We will investigate the summands occurring in those subsums in order to prove
 \begin{equation}
 \label{toprove}
     \sum_{t\in\Int \sigma} e^t = \sum_{\tau\in\Sigma} (-1)^{n-\dim \tau}  \sum_{t\in \tau} e^t.
 \end{equation}
 If we then (using Lemma \ref{lemgl}) denote by $P_\tau$ the polynomial with nonnegative coefficients such that
 $$\sum_{t\in \tau} e^t = (-1)^{\dim\tau} \prod_{i=1}^k \frac{1}{S_{w_i}} \cdot P_\tau(\{S_{w_i}\}_{i=1,2,\dots,K}),$$
 we will obtain
 \begin{multline*}
  \sum_{t\in\Int \sigma} e^t = \sum_{\tau\in\Sigma} (-1)^{n-\dim \tau}  (-1)^{\dim\tau} \prod_{i=1}^k \frac{1}{S_{w_i}} \cdot P_\tau(\{S_{w_i}\}_{i=1,2,\dots,K}) \\
  = (-1)^n \prod_{i=1}^k \frac{1}{S_{w_i}}\cdot \sum_{\tau\in\Sigma} P_\tau(\{S_{w_i}\}_{i=1,2,\dots,K}),
 \end{multline*}
 thus proving the claim for
 $$P(\{S_{w_i}\}_{i=1,2,\dots,N}) = \sum_{\tau\in \Sigma} P_\tau(\{S_{w_i}\}_{i=1,2,\dots,K}).$$
 For any $\kappa\in \Sigma$ let $\is(\kappa) = \sum_{t\in \Int\kappa} e^t$ (abbreviation of \emph{Interior Sum}).
 We have
 $$\Int\sigma = \bigsqcup_{\kappa\in\Sigma: \Int \kappa\subset \Int\sigma} \Int \kappa$$
 On the other hand, for any $\tau\in\Sigma$ we have
 $$\tau = \bigsqcup_{\kappa\in\Sigma : \kappa\subset \tau} \Int\kappa.$$
 Thus by substituting these two equalities, \eqref{toprove} becomes equivalent to
 $$\sum_{\kappa\in\Sigma: \Int \kappa\subset \Int\sigma} \is(\kappa) = \sum_{\tau\in\Sigma} (-1)^{n-\dim \tau}\sum_{\kappa\in\Sigma : \kappa\subset \tau} \is(\kappa). $$
 Now we need to prove that for any $\kappa\in\Sigma$ the value $\is(\kappa)$ occurs with the same multiple on both sides. On the left-hand side the multiple is 1 if $\Int\kappa\subset \Int\sigma$ and 0 otherwise. On the right however, it is equal to $\sum_{\tau\in\Sigma : \kappa\subset\tau} (-1)^{n-\dim\tau}$. Thus we reduced the theorem to proving the equation
 \begin{equation}
 \label{inex}
\sum_{\tau\in\Sigma : \kappa\subset\tau} (-1)^{n-\dim\tau}=
 \begin{cases}
 1 \text{ if }\Int\kappa \subset \Int\sigma, \\
 0 \text{ otherwise}.
 \end{cases}
 \end{equation}
 for any $\kappa\in\Sigma$.

 Now take any such $\kappa$ of dimension $k$. Take any point $p\in\Int\kappa$ and let $A$ be a $n-k$-dimensional disk around $p$, intersecting $\kappa$ transversally. The disk $A$ will then intersect all the cones $\tau\in\Sigma$ containing $\kappa$; by taking $A$ small enough, we can also ensure that it does not intersect any other cones from $\Sigma$. Then $A\cap \sigma$ is a contractible set, as a star domain centered at $p$ (in fact, it's just homeomorphic to $D^{n-k}$). 
 
 Now $p$ is an interior point of $A\cap\sigma$ iff it is an interior point of $\sigma$, which is equivalent to $\Int \kappa\subset \Int \sigma$. The intersections of cones $\tau\supset \kappa$ with $A$ produce a triangulation of $A\cap\sigma$. Every such cone $\tau$ produces a simplex of dimension $\dim\tau-k$. Then
 $$\sum_{\tau\in\Sigma : \kappa\subset\tau} (-1)^{\dim\tau-k}$$
 is the Euler characteristic of the pair $(A\cap\sigma,\partial A\cap \sigma)$ (only the simplices contained in the interior of A 
are counted). But we have
 \begin{equation}
 \label{chieq}
     \chi(A\cap\sigma,\partial A\cap \sigma) = \chi(A\cap\sigma) - \chi(\partial A\cap \sigma) = \chi\left(D^{n-k}\right) - \chi(\partial A\cap \sigma) = 1 - \chi(\partial A\cap \sigma).
 \end{equation}
 Now if $p$ is the interior point of $A\cap\sigma$, we have $\partial A\cap \sigma  = \partial(A\cap \sigma) \simeq S^{n-k-1}$ and otherwise $\partial A\cap \sigma \simeq D^{n-k-1}$. In the former case $\chi(\partial A\cap \sigma) = 1 - (-1)^{n-k-1}$ and in the latter $\chi(\partial A\cap \sigma) = 1$. Combining these with \eqref{chieq} yields
 $$\chi(A\cap\sigma,\partial A\cap \sigma) = 
 \begin{cases}
 (-1)^{n-k} \text{ if } \Int \kappa\subset \Int \sigma, \\
 0 \text{ otherwise.}
 \end{cases}
 $$
 Thus
 \begin{multline*}
 \sum_{\tau\in\Sigma : \kappa\subset\tau} (-1)^{n-\dim\tau}=(-1)^{n-k} \sum_{\tau\in\Sigma : \kappa\subset\tau} (-1)^{\dim\tau-k} \\
 = (-1)^{n-k} \chi(A\cap\sigma,\partial A\cap \sigma) = 
 \begin{cases}
 1 \text{ if } \Int \kappa\subset \Int \sigma, \\
 0 \text{ otherwise,}
 \end{cases}
 \end{multline*}
 and the proof of \eqref{inex}, and hence of \eqref{toprove} is finished.
\end{proof}

\begin{remark}
The equation \eqref{toprove} is essentially the same as the generalized Euler relations \cite[Equations (14.64) and (14.65)]{Schneider} for polytopes and face-to-face mosaics. We provided an elementary proof thereof.
\end{remark}

 Theorem \ref{main} combined with Theorem \ref{formula} implies the following generalization of \cite[Proposition 13.2]{Weber2016}.

\begin{theorem}
\label{final}
Let $X$ be a toric variety defined by a $d$-dimensional cone $\sigma\in N$. Let $w_1$, $w_2$, \dots, $w_N$ be the generators of the dual cone $\sigma^\vee\in M$, with $w_1$, $w_2$, \dots, $w_k$ being the primitive vectors on the rays of $\sigma^\vee$. Then the Hirzebruch class of inclusion of the open orbit $td_y^T(\Omega_{\{0\}}\to\C^K)$ is of the form
\begin{equation} \label{hirzform}
\delta^d \prod_{i=1}^k \frac{1}{S_{w_i}} \cdot P(\{S_{w_i}\}_{i=1,2,\dots,K})[\Omega_\sigma],
\end{equation}
 where $\delta = -y-1$, $P$ is a polynomial with nonnegative coefficients and $S_{w_i} = e^{w_i} - 1$.
\end{theorem}

\begin{example}

As an example, we compute the sum $\sum_{t\in \Int \sigma} e^t$ for $\sigma$ being the cone generated by $P_1 = (0,0,1)$, $P_2 = (1,0,1)$, $P_3 = (1,1,1)$, $P_4 = (0,1,1)$ -- which is a cone over the square $P_1P_2P_3P_4$. This corresponds to the affine variety which is a cone over $\PP^1\times \PP^1$ -- embedded in $\A^4$ (a cone over the Segre embedding).

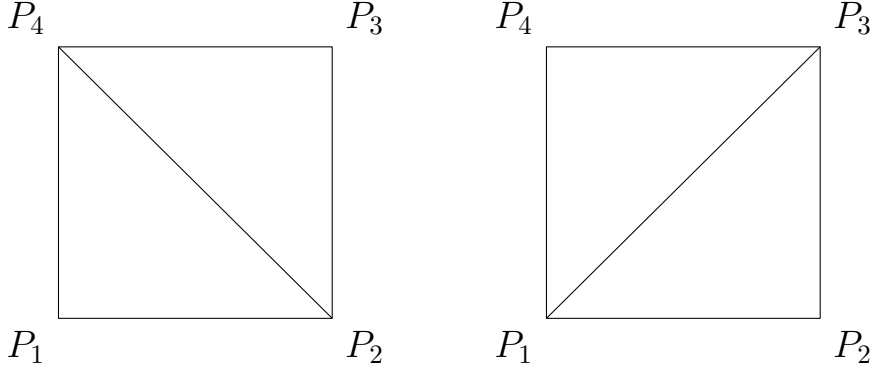
\begin{figure}
\centering
\begin{tikzpicture}[line join=round, scale=0.9]
\LARGE
 \coordinate [label=-135:  $P_1$] (P1) at (0,0);
 \coordinate [label=-45:  $P_2$] (P2) at (4,0);
 \coordinate [label=45:  $P_3$] (P3) at (4,4);
 \coordinate [label=135:  $P_4$] (P4) at (0,4);
 
 \draw (0,0)--(4,0)--(4,4)--(0,4)--cycle;
 \draw (4,0)--(0,4);
\end{tikzpicture}\hspace{1cm}
\begin{tikzpicture}[line join=round, scale=0.9]
\LARGE
 \coordinate [label=-135:  $P_1$] (P1) at (0,0);
 \coordinate [label=-45:  $P_2$] (P2) at (4,0);
 \coordinate [label=45:  $P_3$] (P3) at (4,4);
 \coordinate [label=135:  $P_4$] (P4) at (0,4);
 
 \draw (0,0)--(4,0)--(4,4)--(0,4)--cycle;
 \draw (0,0)--(4,4);
\end{tikzpicture}
\caption{The square $P_1P_2P_3P_4$ with two possible triangulations}
\label{square}
\end{figure}

There are two ways of triangulating the square without adding any rays, they are depicted in Figure \ref{square}. Let us focus on the left one. From \eqref{toprove} we have

\begin{equation}
\label{smdec}
\sum_{t\in\Int \sigma} e^t = 
\sum_{t\in \lambda_1} e^t + \sum_{t\in \lambda_2}e^t
-\sum_{i=1}^4 \sum_{t\in \rho_i} e^t - \sum_{t\in \tau} e^t + \sum_{i=1}^4 \sum_{t\in \N P_i} e^t - 1
\end{equation}
with:

\begin{itemize}
 \item $\lambda_1$ being the closed cone over triangle $P_1P_2P_4$;
 \item $\lambda_2$ being the closed cone over triangle $P_2P_3P_4$;
 \item $\rho_i$ being the closed cone over the segment $P_iP_{i+1}$ (with $P_5 = P_1$);
 \item $\tau$ being the closed cone over the segment $P_2P_4$.
\end{itemize}
Note that we used that $P_i$'s are primitive on the corresponding rays and that the 1 at the end corresponds to the trivial cone.

The cone $\lambda_1$ is actually generated by $P_1$, $P_2$, $P_4$ (as a semigroup) or, in other words, the lattice $\Z P_1 +\Z P_2 + \Z P_3$ is saturated. Indeed, any point in $\lambda_1$ is of the form $\theta_1 P_1 + \theta_2 P_2 + \theta_4 P_4 = (\theta_2,\theta_4,\theta_1+\theta_2+\theta_4)$ for some $\theta_1,\theta_2,\theta_4\in \R_{\ge 0}$. If all the coordinates are to be integers, then clearly $\theta_1,\theta_2,\theta_4$ are integers. Similarly, $P_2$, $P_3$, $P_4$ generate a saturated lattice and hence analogous property holds also for all the smaller cones. Thus we have

$$\sum_{t\in\lambda_1} e^t = \frac{1}{(1-e^{P_1})(1-e^{P_2})(1-e^{P_4})} = - \frac{1}{S_{P_1}S_{P_2}S_{P_4}}, $$
$$\sum_{t\in\lambda_2} e^t = \frac{1}{(1-e^{P_2})(1-e^{P_3})(1-e^{P_4})} = - \frac{1}{S_{P_2}S_{P_3}S_{P_4}}, $$
$$\sum_{t\in\rho_i} e^t = \frac{1}{(1-e^{P_i})(1-e^{P_{i+1}})} = \frac{1}{S_{P_i}S_{P_{i+1}}}, $$
$$\sum_{t\in\tau} e^t = \frac{1}{(1-e^{P_2})(1-e^{P_{4}})} = \frac{1}{S_{P_2}S_{P_{4}}}, $$
$$\sum_{t\in \N Q} e^t = \frac{1}{1-e^Q} = -\frac{1}{S_Q} \text{ for any } Q \in \Z^3.$$
Plugging these into \eqref{smdec} yields

\allowdisplaybreaks
\setlength{\jot}{13pt}
\begin{align*}
\sum_{t\in\Int \sigma} e^t =  - \frac{1}{S_{P_1}S_{P_2}S_{P_4}} - \frac{1}{S_{P_2}S_{P_3}S_{P_4}} - \sum_{i=1}^4 \frac{1}{S_{P_i}S_{P_{i+1}}} -\frac{1}{S_{P_2}S_{P_{4}}} -  \sum_{i=1}^4 \frac{1}{S_{P_i}} -1
\\
=-\frac{S_{P_3}+S_{P_1} + \sum\limits_{i=1}^4 S_{P_i}S_{P_{i+1}}+S_{P_1}S_{P_3} +\sum\limits_{i=1}^4 S_{P_i}S_{P_{i+1}}S_{P_{i+2}}+S_{P_1}S_{P_2}S_{P_3}S_{P_4}}{S_{P_1}S_{P_2}S_{P_3}S_{P_4}}.
\end{align*}
We obtained the form from Theorem \ref{main} for $P(S_{P_1},S_{P_2},S_{P_3},S_{P_4})$ equal to 
$$
 S_{P_1}+S_{P_3}+S_{P_1}S_{P_3} + \sum\limits_{i=1}^4 S_{P_i}S_{P_{i+1}} +\sum\limits_{i=1}^4 S_{P_i}S_{P_{i+1}}S_{P_{i+2}}+S_{P_1}S_{P_2}S_{P_3}S_{P_4}.
$$
The other triangulation would lead to a different polynomial $P'(S_{P_1},S_{P_2},S_{P_3},S_{P_4})$ equal to
$$S_{P_2}+S_{P_4}+S_{P_2}S_{P_4} + \sum\limits_{i=1}^4 S_{P_i}S_{P_{i+1}} +\sum\limits_{i=1}^4 S_{P_i}S_{P_{i+1}}S_{P_{i+2}}+S_{P_1}S_{P_2}S_{P_3}S_{P_4}.$$
By the results above, both polynomials give the same value when we substitute $S_{P_i} = e^{P_i}-1$. We can notice it directly with the following equality:
\begin{multline*}
S_{P_1}+S_{P_3}+S_{P_1}S_{P_3} = (S_{P_1}+1)(S_{P_3}+1) - 1 = e^{P_1}e^{P_3}-1 \\
= e^{(0,0,1)}e^{(1,1,1)} -1  = e^{(1,1,2)} -1  = e^{(1,0,1)}e^{(0,1,1)} -1 \\
= e^{P_2}e^{P_4}-1 = (S_{P_2}+1)(S_{P_4}+1) - 1 = S_{P_2}+S_{P_4}+S_{P_2}S_{P_4}.
\end{multline*}

\end{example}

\bibliographystyle{plain}
\bibliography{ref}

\end{document}